\documentclass[11pt, reqno]{amsart}
\usepackage{times}
\usepackage{amssymb}

\newtheorem{theorem}{Theorem}
\newtheorem{lemma}[theorem]{Lemma}

\newtheorem{corollary}[theorem]{Corollary}

\newtheorem*{thma}{Theorem A}
\newtheorem*{thmb}{Theorem B}
\newtheorem*{thmc}{Theorem C}
\newtheorem*{thmd}{Theorem D}
\newtheorem*{thme}{Theorem E}

\newcommand{\R}{{\mathbb R}}
\newcommand{\C}{{\mathbb C}}
\newcommand{\T}{{\mathbb T}}
\newcommand{\D}{{\mathbb D}}
\newcommand{\inr}{\int_{\R}}
\newcommand{\inc}{\int_{\C}}
\newcommand{\F}{\mathcal{F}}

\DeclareMathOperator{\re}{{\rm Re}}
\DeclareMathOperator{\im}{{\rm Im}}

\allowdisplaybreaks

\begin{document}

\title{Canonical integral operators on the Fock space}

\author{Xingtang Dong}
\address{School of Mathematics, Tianjin University, Tianjin 300350, China.}
\email{dongxingtang@163.com; dongxingtang@tju.edu.cn}

\author{Kehe Zhu}
\address{Department of Mathematics and Statistics, SUNY, Albany, NY 12222, USA.}
\email{kzhu@albany.edu}

\subjclass[2010]{Primary 30H20; Secondary 47G10.}

\keywords{Fock space, linear canonical transforms, Bargmann transform, general linear group,
special linear group.}

\thanks{\noindent Research of Dong was supported in part by the Natural Science Foundation of
Tianjin City of China (Grant No. 19JCQNJC14700). Research of Zhu was supported by NNSF of
China (Grant No. 11720101003).}

\begin{abstract}
In this paper we introduce and study a two-parameter family of integral operators on the Fock space
$F^2(\C)$. We determine exactly when these operators are bounded and when they are unitary.
We show that, under the Bargmann transform, these operators include the classical linear canonical
transforms as special cases. As an application, we obtain a new unitary projective representation
for the special linear group $SL(2,\R)$ on the Fock space.
\end{abstract}

\maketitle

\section{Introduction}

Let $\R$ denote the real line and $\C$ denote the complex plane. We are going to study operator theory on
three particular Hilbert spaces: the Lebesgue space $L^2(\R, dx)$, the weighted Lebesgue space
$L^2(\C, d\lambda)$, and the Fock space
$$F^2=L^2(\C, d\lambda)\cap H(\C),$$
where $H(\C)$ is the space of all entire functions on $\C$ and
$$d\lambda(z)=\frac{1}{\pi}e^{-|z|^2}\,dA(z)$$
is the Gaussian measure on $\C$. Here $dA$ is ordinary area measure. The inner product on $F^2$ is
inherited from $L^2(\C,d\lambda)$.

The Fock space has its origin in mathematical physics. The mathematical theory of Fock spaces has also
experienced a rapid development over the past few decades. In particular, starting from the 1980s, Berger
and Coburn systematically studied Toeplitz and Hankel operators on the Fock space; see \cite{BC1, BC2, BC3}
for example. On the other hand, motivated by problems from engineering, Seip and Wallsten completely
characterized the so-called interpolating and sampling sequences for Fock spaces in \cite{Seip, SW}.
See \cite{Zhu1} for a recent survey of the mathematical theory of Fock spaces.

Although all separable infinite-dimensional Hilbert spaces are isomorphic, there is a particular unitary
transformation $B$ from the Lebesgue space $L^2(\R, dx)$ onto the Fock space $F^2$ that is very
natural and extremely useful. This is the Bargmann transform, which enables us to translate operators
on $L^2(\R, dx)$ to operators on $F^2$, and vice versa. A prominent example is the Fourier transform
$\F$ as a bounded operator on $L^2(\R, dx)$, which, under the Bargmann transform, becomes the
operator $f(z)\mapsto f(iz)$ on $F^2$. It is clear that the latter form of the Fourier transform on $F^2$
has a much better structure (at least on surface) than the original form on $L^2(\R, dx)$. This example
illustrates the enormous potential of the Bargmann transform in operator theory.

Another important example is the Hilbert transform $H$ as a bounded operator on $L^2(\R, dx)$. Under
the Bargmann transform, it becomes the following integral operator on $F^2$:
$$Hf(z)=\inc e^{z\overline w}\varphi(z-\overline w)f(w)\,d\lambda(w),$$
where $\varphi$ is a particular function in $F^2$ which can be written down explicitly. Motivated by this
example, the second author here raised the following question in \cite{Zhu2}: characterize all functions
$\varphi\in F^2$ such that the integral operator above is bounded on $F^2$. A beautiful answer to this
question was given in \cite{CLS}, and more related work was done in \cite{WW}. See \cite{Zhu3} for
many other examples of operators on $L^2(\R, dx)$ and their counterparts, under the Bargmann transform,
on $F^2$.

In this paper, we introduce and study a two-parameter family of integral operators on
$F^2$. More specifically, for any $(s,t)\in\C^2$ with $s\not=0$, we define
$$T^{(s,t)}f(z)=\inc K^{(s,t)}(z,w)f(w)\,d\lambda(w),$$
where
$$K^{(s,t)}(z,w)=\frac1{\sqrt s}\,\exp\left[\frac{tz^2-\overline{tw^2}+2z\overline w}{2s}\right].$$
Here and throughout the paper, the complex square root is defined as follows:
$$\sqrt z=\sqrt{|z|}\,e^{i\theta/2},\qquad z=|z|e^{i\theta},\qquad \theta\in(-\pi, \pi].$$
It is clear that $T^{(s,t)}f$, whenever well-defined, is an entire function. Therefore, when $T^{(s,t)}$
is a bounded operator on $L^2(\C, d\lambda)$, it will also be a bounded operator on $F^2$.
This simple observation will be used many times later without being explicitly mentioned again.

If $f$ is continuous on $\C$ with compact support, it is easy to see that $T^{(s,t)}f$ is a well-defined
entire function for any $s\not=0$. Therefore, $T^{(s,t)}: L^2(\C, d\lambda)\to H(\C)$ is densely defined for
any $s\not=0$. However, when $|t|\ge2|s|$, it is not clear if there is any nonzero function $f\in F^2$
such that the integral defining $T^{(s,t)}f$ is convergent. In particular, for $|t|\ge2|s|$, the integral
for $T^{(s,t)}f$ is divergent for all nonzero polynomials $f$ and for all finite linear combinations of 
(ordinary) kernel functions. Therefore, when considering $T^{(s,t)}$ as operators on $F^2$, we will make 
the natural assumption that $|t|<2|s|$.

Our main results are Theorems A, B, C, D, and E below.

\begin{thma} 
The operators $T^{(s,t)}$ on $F^2$, where $|t|<2|s|$, have the following properties.
\begin{itemize}
\item[(i)] $T^{(s,t)}$ is bounded if and only if $|s|^2\ge|t|^2+1$.
\item[(ii)] $T^{(s,t)}$ is unitary if and only if $|s|^2=|t|^2+1$.
\item[(iii)] When $|s|^2>|t|^2+1$, the operator $T^{(s,t)}$ belongs to the Hilbert-Schmidt class $S_2$ with
$$\|T^{(s,t)}\|^2_{S_2}=\frac{|s|}{\sqrt{|s|^2-|t|^2-1}}.$$
\end{itemize}
\end{thma}

Our next result concerns the general and special real linear groups of order $2$. Let us write
$$GL(\C\times\C)=\{(s,t)\in\C^2: |s|\not=|t|\}$$
and
$$SL(\C\times\C)=\{(s,t)\in\C^2: |s|^2=|t|^2+1\}.$$
We show that these are the complex versions of the general and special real linear groups, respectively.
We also show that the special linear group of order $2$ can be represented as unitary operators on the
Fock space.

\begin{thmb}
Let $GL(2,\R)$ denote the general real linear group and $SL(2, \R)$ the special real linear group of
order $2$.
\begin{itemize}
\item[(i)] The set $GL(\C\times\C)$ is a group with the operation
$$(s_1, t_1)\cdot(s_2,t_2)=(s_1s_2+\overline t_1t_2, t_1s_2+\overline s_1t_2).$$
Moreover, the set $SL(\C\times\C)$ is a subgroup of $GL(\C\times\C)$.
\item[(ii)] The mapping $\varphi: GL(2, \R)\to GL(\C\times\C)$ defined by
$$\varphi\begin{bmatrix} a & b\\ c & d\end{bmatrix}=\left(\frac{a+ib+d-ic}2,\,\frac{a+ib-d+ic}2\right)$$
is a group isomorphism. Moreover, $\varphi$ maps $SL(2, \R)$ onto $SL(\C\times\C)$.
\item[(iii)] The mapping $(s,t)\mapsto T^{(s,t)}$ is a unitary projective representation of the group
$SL(\C\times\C)$ on $F^2$.
\end{itemize}
\end{thmb}

Recall from mathematical physics that, for every matrix $A=\begin{bmatrix} a & b\\ c & d\end{bmatrix}$
in $SL(2, \R)$, the linear canonical transform
$$\F^A: L^2(\R, dx)\to L^2(\R, dx)$$
is defined by
$$\F^Af(x)=\frac1{\sqrt{i\pi b}}\,e^{idx^2/b}\inr e^{-i(2xt-at^2)/b}f(t)\,dt$$
when $b\not=0$. If $b=0$, then $\F^A$ is defined by
$$\F^Af(x)=\sqrt d\,e^{icdx^2}\,f(dx).$$
Note that $d$ here is a real number but not the differential. Also, the parameter $c$ does not appear in
the definition of $\F^Af$ for $b\not=0$ above because the four parameters are subject to the
condition $ad-bc=1$, so $c=(ad-1)/b$ in this case.

The linear canonical transforms are widely used in science, engineering, and mathematical physics.
Special forms of them include fractional Fourier transforms, scaling (or dilation), the Fresnel transform,
and chirp multiplication. See \cite{BM2, HKOS, OZK, W} for more information about linear canonical
transforms.

\begin{thmc}
Under the Bargmann transform, the linear canonical transforms become the operators $T^{(s,t)}$ on
$F^2$, where $(s,t)\in SL(\C\times\C)$.
\end{thmc}

Recall that the Fourier transform $\F: L^2(\R, dx)\to L^2(\R, dx)$ becomes the operator $\F: F^2\to F^2$,
where $\F(f)(z)=f(iz)$. Using this form and with the help of the canonical monomial orthonormal basis
for $F^2$, it is easy to show that the spectrum of the unitary operator $F$ on $L^2(\R, dx)$ or on $F^2$
consists of four points on the unit circle: $i^k$, $k=0,1,2,3$. Each point is an eigenvalue and the
corresponding eigenspace can be described in terms of the Hermite polynomials in the case of
$L^2(\R, dx)$ and in terms of Taylor coefficients in the case of $F^2$. See \cite{Zhu3} again. The following
result shows that much of this can be generalized to linear canonical transforms.

\begin{thmd}
Let $(s,t)\in SL(\C\times\C)$ with $|\re s|<1$. Then there exists a number $\gamma$ in the open unit disk
such that for each $n=0,1,2,\cdots$ the number
$$\lambda_n=\frac1{\sqrt s}\,\sqrt{\frac s{s+\overline t\gamma}}\,\frac1{(s+\overline t\gamma)^n}$$
is an eigenvalue of $T^{(s,t)}$, and a corresponding eigenvector is the function
$$f_n(z)=e^{\gamma z^2/2}\inr H_n\left(\frac x\rho\right)\exp\left[-\frac2{1+\gamma}
\left(x-\frac{1+\gamma}2\,z\right)^2\,\right]\,dx,$$
where $H_n$ is the classical Hermite polynomial of degree $n$ and $\rho$ is a certain positive number.
\end{thmd}

In order to prove Theorem D, we will first obtain some properties of Hermite polynomials, which appear to
be new (to the best of our knowledge). In particular, we find out that each Hermite polynomial is a
solution of a certain integral equation, which enables us to obtain the classical closed-form formula
for Hermite polynomials. More specifically, we have the following.

\begin{thme}
Let $\mu,\nu, a, b\in\C$ with $a\not=0$, $b\neq 0$, $\re \mu>0$, $\re \nu>0$, and
$\nu b^2\neq\mu a^2$. Suppose
$$P_n(x)=c_0x^n+c_1x^{n-1}+\cdots+c_n,\qquad c_0\not=0.$$
If $a^{k}\neq b^{k}$ for $1\le k\leq n$, then the following conditions are equivalent:
\begin{itemize}
  \item[(a)] $P_n(x)$ is a constant multiple of the Hermite polynomial $H_n(x)$.
  \item[(b)] $P_n(x)$ is a solution of the integral equation
$$\inr P_n(x)e^{-\mu \left(x-az\right)^2}dx=C_n \inr P_n(x)e^{-\nu \left(x-bz\right)^2}dx,$$
where
\begin{equation}\label{eqHermieeqN}
C_n=\frac{\sqrt{\nu}}{\sqrt{\mu}}\left(\frac{a}{b}\right)^n,\qquad
\frac{(b^2-a^2)\mu\nu}{\nu b^2-\mu a^2}=1.
\end{equation}
  \item[(c)] $c_{2k+1}=0$ whenever $2k+1\leq n$, and
$$(n-2k)!\,c_{2k}=-2(2k+2)(n-2k-2)!\,c_{2k+2}$$
whenever $2k+2\leq n$. Or equivalently,
$$P_n(x)=C\sum_{k=0}^{[\frac{n}{2}]}\frac{(-1)^k2^nn!}{4^kk!(n-2k)!}x^{n-2k}$$
for some nonzero constant $C$. Here $[\cdot]$ denotes the greatest integer function.
\end{itemize}
\end{thme}

We will actually prove a more general result about the solution of such an integral equation without
the assumptions in \eqref{eqHermieeqN}. We will also evaluate the integrals in part (b) above.

\section{A two-parameter family of integral operators}

In this section we study the integral operators $T^{(s,t)}$, $(s,t)\in\C^2$ with $s\not=0$, on the Fock space
$F^2$. We will determine exactly when these operators are bounded and when they are unitary. Recall that
$T^{(s,t)}$ is the integral operator (with respect to the Gaussian measure) whose integral kernel is given by
$$K_w^{(s,t)}(z)=K^{(s,t)}(z,w)=\frac1{\sqrt s}\,\exp\left[\frac{tz^2-\overline{tw^2}+2z\overline w}{2s}
\right].$$
We begin with some elementary properties of these kernel functions.

\begin{lemma}\label{lem1}
The two-parameter kernel functions $K^{(s,t)}(z,w)$ have the following properties.
\begin{enumerate}
\item[(a)] $\overline{K^{(s,t)}(z,w)}=cK^{(\,\overline{s},-{t})}(w,z)$, where $c=1$ when $s+|s|\not=0$
and $c=-1$ when $s<0$. 
\item[(b)] $K_w^{(s,t)}\in F^2$ if and only if $|s|>|t|$.
\end{enumerate}
\end{lemma}

\begin{proof}
Part (a) is obvious with
$$c=\frac{\sqrt{\overline s}}{\overline{\sqrt s}}=\begin{cases} 1, &s+|s|\not=0,\\ -1, & s<0.\end{cases}$$

If $t=0$, then it is clear that both $|s|>|t|$ and $K_w^{(s,t)}\in F^2$ are true. If $t\neq0$, the entire
function $K_w^{(s,t)}$ is of order $2$ and type $|t|/(2|s|)$, so by Theorem 2.12 of \cite{Zhu1},
$K_w^{(s,t)}\in F^2$ if and only if the type $|t|/(2|s|)$ is less than $1/2$, or $|t|<|s|$. This
proves part (b).
\end{proof}

\begin{lemma}\label{lem2}
For $k=1,2$ let $(s_k, t_k)\in\C^2$ with $|s_k|>|t_k|$. Then
\begin{multline*}
\inc K^{(s_1,t_1)}(z,w)K^{(s_2,t_2)}(w,u)\,d\lambda(w)=\\
cK^{(s,t)}(z,u)\exp\left[\frac{t_2(|t_1|^2+1-|s_1|^2)}{2s_1(s_1s_2+\bar t_1t_2)}\,z^2-
\frac{\bar t_1(|t_2|^2+1-|s_2|^2)}{2s_2(s_1s_2+\bar t_1t_2)}\,\overline u^2\right],
\end{multline*}
where $(s, t)=(s_1s_2+\bar t_1t_2,\, t_1s_2+\bar s_1t_2)$, $|s|>|t|$, and
$$c=\frac{\sqrt{s_1s_2+{\overline{t_1}t_2}}}{\sqrt{s_{1}}\sqrt{s_{2}}}
\sqrt{\frac{s_{1}s_{2}}{s_1s_2+{\overline{t_1}t_2}}}=\pm1.$$
\end{lemma}

\begin{proof}
It follows from Lemma~\ref{lem1} and H\"older's inequality that the integral in question exists for all $z$
and $u$ in $\C$. The desired integral formula then follows from the elementary identity
(see (1.18) of \cite{B1} for example),
\begin{align}\label{eqIC}
&\inc\exp\left[\frac{\gamma}{2}\,w^{2}+aw+\frac{\overline\delta}{2}\,\overline{w}^{2}
+\overline{b}\,\overline{w}\right]\,d\lambda(w)\nonumber\\
&\qquad\qquad\qquad\qquad\qquad=\frac{1}{\sqrt{1-\gamma\overline{\delta}}}
\exp\left[{\frac{\overline{\delta}a^{2}+\gamma \overline{b}^{2}+2a\overline{b}}
{2(1-\gamma\overline{\delta})}}\right]
\end{align}
under the assumption $|\gamma+\delta|^{2}<4$, which implies that
$$\re(1-\gamma\overline\delta)=1-\frac14|\gamma+\delta|^2+|\gamma-\delta|^2>0.$$
It is easy to check that
$$|s|^2-|t|^2=(|s_1|^2-|t_1|^2)(|s_2|^2-|t_2|^2)>0.$$
This proves the lemma.
\end{proof}

\begin{corollary}\label{cor1}
If $(s,t)\in\C^2$ with $|s|>|t|$, then
$$\left\|K^{(s,t)}_w\right\|=\frac{1}{\sqrt[4]{|s|^2-|t|^2}}\exp\left[\frac{|{w}|^2}{2(|s|^2-|t|^2)}\right]
\left|\exp\left[\frac{t(|t|^2+1-|s|^2)}{2\overline{s}(|s|^2-|t|^2)}{w^2}\right]\right|.$$
\end{corollary}

\begin{proof}
This is a direct consequence of Lemmas \ref{lem1} and \ref{lem2}. We leave the elementary details to
the interested reader.
\end{proof}

We now begin to study the two-parameter family of operators $T^{(s,t)}$. First note that for $|s|>|t|$
the function $T^{(s, t)}f$ is well-defined and entire for every $f\in L^2(\C,d\lambda)$. In fact,
by H\"{o}lder's inequality and part (b) of Lemma~\ref{lem1}, we have
\begin{align*}
    \left|T^{(s,t)}f(z)\right|&\leq \frac{1}{\sqrt{|s|}}\|f\|\left[\inc \left|e^{\frac{t}{2s}z^2-\frac{\overline{t}}{2s}\overline{w}^2+\frac{ z\overline{w}}{s}}\right|^2d\lambda(w)\right]^\frac{1}{2}\\
    &=\|f\|\ \|K_z^{(\overline{s},-{t})}\|<\infty.
\end{align*}
For any $w\in\C$ let us write
$$K_w(z)=K(z, w)=K^{(1,0)}(z, w)=e^{z\overline{w}},\qquad z\in\C,$$ which is the well known reproducing
kernel of $F^2$ at $w$.

\begin{lemma}\label{lem3}
Let $(s,t)\in\C^2$ with $|s|>|t|$ and $u\in\C$. Then the operator $T^{(s,t)}$ has the following properties.
\begin{enumerate}
\item[(a)] $T^{(s,t)}K_u=K_u^{({s},{t})}$.
\item[(b)] If $|s|^2=|t|^2+1$, then $T^{(s,t)}K_u^{(\overline{s},-{t})}=cK_u$, where
$c=\pm1$ as in Lemma~\ref{lem1}.
\end{enumerate}
\end{lemma}

\begin{proof}
It follows from Lemma~\ref{lem2} that
$$T^{(s,t)}K_u(z)=\inc K^{(s,t)}(z,w) K^{(1,0)}(w, u)\,d\lambda(w)=K_u^{({s},{t})}(z).$$
If $|s|^2=|t|^2+1$, Lemma~\ref{lem2} again gives
\begin{eqnarray*}
T^{(s,t)}K_u^{(\overline{s},-{t})}(z)=\inc K^{(s,t)}(z,w) K^{(\overline{s},-{t})}(w, u)\,d\lambda(w)
=c K_u^{(1,0)}(z),
\end{eqnarray*}
where $c=|s|/(\sqrt{s}\sqrt{\overline{s}})$. This proves the desired result.
\end{proof}

We can now prove the first main result of the paper, Theorem A, which is stated in a slightly different way
as follows.

\begin{theorem}\label{thmTB}
The integral operators $T^{(s,t)}$ on $F^2$, where $|t|<2|s|$, have the following properties.
\begin{enumerate}
\item[(a)] If $|s|^2-|t|^2<1$, then $T^{(s,t)}$ is unbounded on $F^2$.
\item[(b)] If $|s|^2-|t|^2>1$, then $T^{(s,t)}$ is not only bounded on $F^2$ but also in the
Hilbert-Schmidt class $S_2$ with
$$\|T^{(s,t)}\|_{S_2}=\frac{\sqrt{|s|}}{\sqrt{|s|^2-|t|^2-1}}.$$
\item[(c)] If $|s|^2-|t|^2=1$, then $T^{(s, t)}$ is a unitary operator on $F^2$ with
$$\left[T^{(s, t)}\right]^{-1}=\left[T^{(s, t)}\right]^{*}=cT^{(\overline{s}, -t)},$$
where $c=\pm1$ as in Lemma~\ref{lem1}.
\end{enumerate}
\end{theorem}

\begin{proof}
If $|s|\le|t|<2|s|$, then for the constant function $f=1$ in $F^2$ we have
$$T^{(s,t)}f(z)=\frac1{\sqrt s}\,e^{tz^2/(2s)},$$
which does not belong to $F^2$, because it is a nonzero function of order $2$ and type greater than or
equal to $1/2$. See Theorem 2.12 of \cite{Zhu1}. This shows that $T^{(s,t)}$ is unbounded on $F^2$
when $|s|\le|t|<2|s|$.

Next we assume $0<|s|^2-|t|^2<1$. By Lemma~\ref{lem3} and Corollary~\ref{cor1}, we have
\begin{multline*}
    \|T^{(s,t)}K_u\|=\|K_u^{({s},{t})}\|=\\
    \frac{1}{\sqrt[4]{|s|^2-|t|^2}}
    \exp\left[\frac{|{u}|^2}{2(|s|^2-|t|^2)}\right]
    \left|\exp\left[\frac{t(|t|^2+1-|s|^2)}{2\overline{s}(|s|^2-|t|^2)}{u^2}\right]\right|.
\end{multline*}
If $T^{(s, t)}$ is bounded on $F^2$, then there would exist a constant $C>0$ such that
$$\|T^{(s,t)}K_u\|\leq C\|K_u\|=C e^{|u|^2/2},\qquad u\in\C.$$
With the change of variables $u=\sqrt{\frac{2(|s|^2-|t|^2)}{|t|^2+1-|s|^2}}\,z$, this would give
$$|e^{tz^2/\overline s}|\leq C\, \sqrt[4]{|s|^2-|t|^2}\, e^{-|z|^2},\qquad z\in\C,$$
which is clearly impossible. This together with the conclusion in the previous paragraph proves (a).

If $|s|^2-|t|^2>1$, it follows from Corollary~\ref{cor1} that
\begin{multline*}
\inc\inc |K^{(s,t)}(z,w)|^2\,d\lambda(z)\,d\lambda(w)\\
=\frac{1}{\sqrt{|s|^2-|t|^2}}\inc\left|\exp\left[\frac{t(|t|^2+1-|s|^2)}{2\overline{s}(|s|^2-|t|^2)}{w^2}
\right]\right|^2\exp\left[\frac{|{w}|^2}{|s|^2-|t|^2}\right]\,d\lambda(w).
\end{multline*}
With the change of variables $u=\sqrt{\frac{|s|^2-|t|^2-1}{|s|^2-|t|^2}}\,w$ and the help of \eqref{eqIC},
we obtain
\begin{align*}
\inc\inc |K^{(s,t)}(z,w)|^2\,d\lambda(z)\,d\lambda(w)
&=\frac{\sqrt{|s|^2-|t|^2}}{|s|^2-|t|^2-1}
\inc\left|e^{-tu^2/(2\overline s)}\right|^2\,d\lambda(u)\\
&=\frac{|s|}{|s|^2-|t|^2-1}<\infty.
\end{align*}
Therefore, $T^{(s,t)}$ belongs to the Hilbert-Schmidt class $S_2$ and part (b) is proved.

To prove (c) we assume $|s|^2-|t|^2=1$. We will actually prove the boundedness of $T^{(s, t)}$
on $L^2(\C,d\lambda)$. We do this with the help of Schur's test; see \cite[Lemma 2.14]{Zhu1}.

The boundedness of $T^{(s,t)}$ on $L^2(\C, d\lambda)$ will follow if we can show that the
following operator, whose kernel is positive (which is necessary for Schur's test), is bounded
on $L^2(\C, d\lambda)$:
$$Q^{(s, t)}f(z)=\inc\left|K^{(s,t)}(z,w)\right|f(w)\,d\lambda(w).$$
To this end, we consider the positive function $h(z)=e^{|z|^2/4}$.
Making a change of variables and using \eqref{eqIC} again, we get
\begin{align*}
    \inc&\left|K^{(s,t)}(z,w)\right|h(w)^2\,d\lambda(w)\\
    &\qquad=\frac{1}{\sqrt{|s|}}\inc \left|\exp\left[\frac{t}{4s}z^2-\frac{\overline{t}}{4s}\overline{w}^2
    +\frac{ z\overline{w}}{2s}\right]\right|^2e^{\frac{|w|^2}{2}}\,d\lambda(w)\\
    &\qquad=\frac{{2}}{\sqrt{|s|}}\left|\exp\left(\frac{t}{4s}z^2\right)\right|^2\inc \left|\exp\left[-\frac{\overline{t}}{2s}\overline{u}^2+\frac{ z\overline{u}}{\sqrt{2}s}\right]\right|^2\,d\lambda(u)\\
    &\qquad=2\sqrt{|s|}\left|\exp\left(\frac{t}{4s}z^2\right)\right|^2\exp\left[-\frac{\overline{t}}{4\overline{s}}\overline{z}^2-\frac{t}{4s}z^2+\frac{|z|^2}{2}\right]\\
    &\qquad=2\sqrt{|s|}\,h(z)^2.
\end{align*}
Similarly,
\begin{align*}
    \inc&\left|K^{(s,t)}(z,w)\right|h(z)^2\,d\lambda(z)\\
    &\qquad=\frac{1}{\sqrt{|s|}}\inc \left|\exp\left[\frac{t}{4s}z^2-\frac{\overline{t}}{4s}\overline{w}^2+\frac{ z\overline{w}}{2s}\right]\right|^2e^{\frac{|z|^2}{2}}\,d\lambda(z)\\
    &\qquad=\frac{{2}}{\sqrt{|s|}}\left|\exp\left(-\frac{\overline{t}}{4s}\overline{w}^2\right)\right|^2\inc
    \left|\exp\left[\frac{{t}}{2s}{u}^2+\frac{\overline{w }u}{\sqrt{2}s}\right]\right|^2d\,\lambda(u)\\
    &\qquad=2\sqrt{|s|}\,h(w)^2.
\end{align*}
Therefore, Schur's test tells us that $Q^{(s, t)}$ is bounded on $L^2(\C,d\lambda)$ with its norm not
exceeding $2\sqrt{|s|}$.

We next focus on the case of $F^2$. It follows from Lemma~\ref{lem3} that
$$T^{(s, t)}cT^{(\overline{s},-t)} K_u= cT^{(s, t)}K_u^{(\overline{s},-t)}=K_u$$
and
$$cT^{(\overline{s},-t)}T^{(s, t)} K_u= cT^{(\overline{s},-t)}K_u^{(s, t)}=K_u.$$
Since the set of finite linear combinations of kernel functions is dense in $F^2$, we conclude that
$$T^{(s, t)}cT^{(\overline{s},-t)}=cT^{(\overline{s},-t)}T^{(s, t)}=I.$$
By Lemma~\ref{lem1}, the adjoint operator of $T^{(s, t)}$ is given by
$$\left[T^{(s, t)}\right]^{*}f(z)=\inc \overline{K^{(s,t)}(w,z)}f(w)\,d\lambda(w)=cT^{(\overline{s}, -t)}f(z).$$
Thus $T^{(s, t)}$ is unitary on $F^2$. This completes the proof of the theorem.
\end{proof}

\section{Eigenvalues and eigenvectors of $T^{(s,t)}$}

Let $\D=\{z\in\C: |z|<1\}$ be the unit disk in the complex plane and let $\T=\partial\D$ be the unit circle.
For any fixed $(s, t)\in\C^2$ with $|s|^2=1+|t|^2$, the operator $T^{(s, t)}$ is unitary, so its spectrum
$\sigma(T^{(s,t)})$ is contained in $\T$. When $|\re s|<1$, we will find a sequence of eigenvalues for
$T^{(s,t)}$ together with certain corresponding eigenfunctions. In some special cases, this will imply
that $\sigma(T^{(s,t)})=\T$.

We first observe that for $f(z)=e^{\gamma z^2/2}$, where $\gamma\in\D$, we have
\begin{align*}
    T^{(s,t)}f(z)&=\frac{1}{\sqrt{s}}\inc\exp\left[\frac{\gamma}{2}w^2+
    \frac{t}{2s}z^2-\frac{\overline{t}}{2s}\overline{w}^2+\frac{ z\overline{w}}{s}\right]\,d\lambda(w)\\
    &=\frac{1}{\sqrt{s}}\,\sqrt{\frac{s}{s+\gamma{\overline{t}}}}\,\exp\left[\frac{t}{2s}z^2
    +\frac{ \gamma z^{2}}{2s(s+\gamma{\overline{t}})}\right].
\end{align*}
Therefore, if $\gamma$ is a solution of the equation
\begin{equation}
s\overline{t}\gamma^2+(s^2-1-|t|^2)\gamma-st=0,
\label{choiceofgamma}
\end{equation}
then the unitary operator $T^{(s, t)}$ on $F^2$ has an eigenvalue
$$\lambda_0=\frac{1}{\sqrt{s}}\,\sqrt{\frac{s}{s+\gamma{\overline{t}}}}$$
with the function $e^{\gamma z^2/2}$ as a corresponding eigenvector.

\begin{lemma}\label{existenceofgamma}
Let $(s,t)\in\C^2$ with $|s|^2=1+|t|^2$. Then the equation (\ref{choiceofgamma}) has a
unique solution $\gamma$ in $\D$ if and only if $|\re s|<1$.
\end{lemma}

\begin{proof}
If $t=0$, then $|s|=1$ and the equation (\ref{choiceofgamma}) becomes $(s^2-1)\gamma=0$, which
has a unique solution $\gamma=0\in\D$ if and only if $s^2-1\not=0$. It is clear that, under the
assumption $|s|=1$, the condition $s^2\not=1$ is equivalent to $|\re s|<1$.

If $t\neq0$, then by the quadratic formula, the solutions of (\ref{choiceofgamma}) are given by
$$\gamma=\frac{-(s-\overline{s})\pm\sqrt{(s+\overline{s})^2-4}}{2\overline{t}}.$$
Writing $s=x+iy$, we have
$$\gamma=\frac{-iy\pm\sqrt{x^2-1}}{\overline t},\qquad
|\gamma|^2=\frac{\left|-iy\pm\sqrt{x^2-1}\right|^2}{|t|^2},$$
and
$$s+\gamma\overline t=x\pm\sqrt{x^2-1}.$$
If $|\re s|\ge1$, then both solutions of (\ref{choiceofgamma}) satisfy
$$|\gamma|^2=\frac{y^2+x^2-1}{|t|^2}=\frac{|s|^2-1}{|t|^2}=1,$$
so (\ref{choiceofgamma}) does not have a solution in the unit disk.
We now consider the case when $|\re s|<1$. Since $|s|^2=x^2+y^2>1$,
it is easy to check that the solution
\begin{equation}\label{eqgamma}
    \gamma_1=\frac{-y+  {\rm sgn} (y)\sqrt{1-x^2}}{\overline{t}}\,i
\end{equation}
satisfies
$$|\gamma_1|^2=\frac{y^2+1-x^2-2|y|\sqrt{1-x^2}}{x^2+y^2-1}<1.$$
Also,
$$s+\gamma_1\overline t=x+i\,{\rm sgn}(y)\sqrt{1-x^2},\qquad
|s+\gamma_1\overline t|=1.$$
It is also easy to check that the other solution
$$\gamma_2=\frac{-y-{\rm sgn}(y)\sqrt{1-x^2}}{\overline t}\,i$$
satisfies $|\gamma_2|\ge1$. This completes the proof of the lemma.
\end{proof}

When $|s|^2=1+|t|^2$ with $|\re s|<1$, we will modify $\lambda_0$ above to obtain additional eigenvalues
for the unitary operator $T^{(s,t)}$ on $F^2$. We will also obtain a corresponding eigenvector for each such
eigenvalue. To this end, we will need a 
new characterization of the classical Hermite polynomials.

\begin{lemma}\label{lemcal}
Let $\mu\in\C$ with $\re \mu>0$ and let $n$ be a nonnegative integer. Then
$$\int_{\R} x^{n}e^{-\mu(x+z)^{2}}\,dx=c_{0}z^{n}+c_{1}z^{n-1}+\cdots+c_{n}$$
for all $z\in\C$, where $c_k=0$ when $k$ is odd and
$$c_k=(-1)^{n-k}\frac{n!\,\Gamma(\frac{k+1}{2})}{k!\,(n-k)!\,(\sqrt{\mu})^{k+1}}$$
when $k$ is even. 
\end{lemma}

\begin{proof}

Consider the integral
$$I(z)=\inr x^{n}e^{-u(x+z)^{2}}\,dx, \qquad z\in\C.$$
Since $\re \mu>0$, it is clear that $I(z)$ is an entire function. Moreover, differentiating under the integral
sign, we obtain
\begin{eqnarray*}
I'(z)&=&\inr x^{n}\frac{\partial}{\partial z}e^{-u(x+z)^2}\,dx\\
&=&\inr x^{n}\frac{\partial}{\partial x}e^{-u(x+z)^2}\,dx\\
&=&-n\inr x^{n-1}e^{-u(x+z)^{2}}\,dx.
\end{eqnarray*}
Repeat this process, we will then get
$$I^{(j)}(z)=n(n-1)\cdots(n-j+1)(-1)^{j}\inr x^{n-j} e^{-u(x+z)^{2}}\,dx$$
for $j=2,3,\cdots,n$. Combining this with Lemma 2 of \cite{DZ}, we obtain
$$I^{(n)}(z)=(-1)^{n}n!\inr e^{-u(x+z)^{2}}\,dx=(-1)^{n}n!\,\frac{\sqrt{\pi}}{\sqrt{u}}$$
for all $z\in\C$.
From this, we deduce that $$I(z)=c_{0}z^{n}+c_{1}z^{n-1}+\cdots+c_{n-1}z+c_{n}$$
with
$$c_k=\frac{I^{(n-k)}(0)}{(n-k)!}
=(-1)^{n-k}\frac{n!}{k!(n-k)!}\inr x^{k} e^{-ux^{2}}\,dx$$
for any $k=0,\cdots,n$. This proves the desired result.
\end{proof}

We are now ready to prove the following interesting result which is more general than Theorem E.

\begin{theorem}\label{thmHit}
Let $\mu,\nu, a, b\in\C$ with $a\not=0$, $b\neq 0$, $\re \mu>0$, $\re \nu>0$, and
$\nu b^2\neq\mu a^2$. Suppose
\begin{equation}\label{eqHermieeq}
\delta=\frac{(b^2-a^2)\mu\nu}{\nu b^2-\mu a^2}
\end{equation}
and
$$P_n(x)=c_0x^n+c_1x^{n-1}+\cdots+c_n,\qquad c_0\not=0.$$
If $a^{k}\neq b^{k}$ for $1\le k\leq n$, then the following conditions are equivalent:
\begin{itemize}
  \item[(a)] $P_n(x)$ is the result of the recursion
  $$P_k(x)=2xP_{k-1}(x)-\frac{1}{\delta} P^{\prime}_{k-1}(x),\qquad 1\le k\le n,$$
  where $P_0(x)=C$ is a nonzero constant. 
  \item[(b)] $P_n(x)$ is a solution of the integral equation
\begin{equation}\label{eqHermite}
    \inr P_n(x)e^{-\mu \left(x-az\right)^2}dx=C_n \inr P_n(x)e^{-\nu \left(x-bz\right)^2}dx
\end{equation}
for some constant $C_n$.
  \item[(c)] $c_{2k+1}=0$ whenever $2k+1\leq n$, and
$$(n-2k)!\,c_{2k}=-2\delta(2k+2)(n-2k-2)!\,c_{2k+2}$$
whenever $2k+2\leq n$.
\end{itemize}
Furthermore, in the cases above, we have
$$C_n=\frac{\sqrt{\nu}}{\sqrt{\mu}}\left(\frac{a}{b}\right)^n$$ and
\begin{equation}\label{eqDeHe}
    P_n(x)=C\sum_{k=0}^{[\frac{n}{2}]}\frac{(-1)^k2^nn!}{4^kk!\,(n-2k)!\,\delta^k} x^{n-2k}
\end{equation}
with $C =c_0 2^{-n}\neq 0$. 
\end{theorem}

\begin{proof}
We begin with some elementary calculations. Differentiating under the integral sign and using integration
by parts, we obtain
\begin{align}\label{eqDP}
    \frac{d}{dz}\left(\inr P_n(x)e^{-\mu \left(x-az\right)^2}dx\right)&=2\mu a\inr \left(x-az\right)P_n(x)e^{-\mu \left(x-az\right)^2}dx\nonumber\\
    &=a\inr P^{\prime}_n(x)e^{-\mu \left(x-az\right)^2}dx.
\end{align}
Continuing this process $k\leq n$ times, we obtain
\begin{equation}\label{eqDk}
    \frac{d^k}{dz^k}\left(\inr P_n(x)e^{-\mu \left(x-az\right)^2}dx\right)
    =a^k\inr P^{(k)}_n(x)e^{-\mu \left(x-az\right)^2}dx.
\end{equation}
It follows from \eqref{eqDP} that
\begin{multline*}
    \inr 2x P_n(x)e^{-\mu \left(x-az\right)^2}dx\\
    =\frac{1}{\mu}\inr P^{\prime}_n(x)e^{-\mu \left(x-az\right)^2}dx
    +2az\inr P_n(x)e^{-\mu \left(x-az\right)^2}dx,
\end{multline*}
which implies that
 \begin{multline}\label{eqDPP}
    \inr \left[2xP_n(x)-\frac{1}{\delta}P^{\prime}_n(x)\right]e^{-\mu \left(x-az\right)^2}dx\\
    =\frac{\delta-\mu}{\mu\delta}\inr P^{\prime}_n(x)e^{-\mu \left(x-az\right)^2}dx+2az\inr P_n(x)e^{-\mu \left(x-az\right)^2}dx
\end{multline}
whenever $\delta\neq 0$.

To show that (a) implies (b), we use induction on $n$. When $n=0$, it is clear from Lemma~\ref{lemcal}
that $P_0(x)=C$, where $C$ is any nonzero constant, is a solution of equation \eqref{eqHermite} with $C_0={\sqrt{\nu}}/{\sqrt{\mu}}$.
Similarly, $P_1(x)=2Cx$ is a solution of equation \eqref{eqHermite} with $$C_1=\frac{\sqrt{\nu}}{\sqrt{\mu}}\frac{a}{b}.$$
So fix $n\geq 1$ and assume that the $n$th degree polynomial
$P_n(x)$ is a solution of equation \eqref{eqHermite}. Then we consider the $(n+1)$st degree polynomial
$$P_{n+1}(x)=2xP_n(x)-\frac{1}{\delta}P^{\prime}_n(x).$$
Differentiate both sides of \eqref{eqHermite}
and use \eqref{eqDP} to obtain
$$\inr P^{\prime}_n(x)e^{-\mu \left(x-az\right)^2}dx=\frac{bC_n}{a} \inr P^{\prime}_n(x)e^{-\nu \left(x-bz\right)^2}dx.$$
Combining the above identity with \eqref{eqHermite} and \eqref{eqDPP}, we obtain
\begin{multline*}
    \inr P_{n+1}(x) e^{-\mu \left(x-az\right)^2}dx\\
    =\frac{b C_n(\delta-\mu)}{a\mu\delta}\inr P^{\prime}_n(x)e^{-\nu \left(x-bz\right)^2}dx+2aC_n z\inr P_n(x)e^{-\nu \left(x-bz\right)^2}dx.
\end{multline*}
By an argument similar to that given for \eqref{eqDPP} we get
\begin{multline*}
    \inr P_{n+1}(x) e^{-\nu \left(x-bz\right)^2}dx\\
    =\frac{\delta-\nu}{\nu\delta}\inr P^{\prime}_n(x)e^{-\nu \left(x-bz\right)^2}dx+2b z\inr P_n(x)e^{-\nu \left(x-bz\right)^2}dx.
\end{multline*}
Clearly, if we define $C_{n+1}$ by $aC_n=C_{n+1}b$, or equivalently,
$$C_{n+1}=\frac{\sqrt{\nu}}{\sqrt{\mu}}\left(\frac{a}{b}\right)^{n+1}$$ by induction, then it follows from \eqref{eqHermieeq} that
$$\frac{bC_n(\mu-\delta)}{a\mu}=\frac{C_{n+1}(\nu-\delta)}{\nu }.$$
Therefore,
$$\inr P_{n+1}(x)e^{-\mu \left(x-az\right)^2}dx=C_{n+1} \inr P_{n+1}(x)e^{-\nu \left(x-bz\right)^2}dx.$$
This completes the induction argument and proves that condition (a) implies (b).

To prove that (b) implies (c), we differentiate both sides of \eqref{eqHermite} $n$ times. It follows from \eqref{eqDk} and Lemma~\ref{lemcal} that
$$a^n n!\,c_0\,\frac{\sqrt{\pi}}{\sqrt{\mu}}=b^n n!\,c_0\,C_n\,\frac{\sqrt{\pi}}{\sqrt{\nu}},$$
which gives
\begin{equation}\label{eqCn}
    C_n=\frac{\sqrt{\nu}}{\sqrt{\mu}}\left(\frac{a}{b}\right)^n.
\end{equation}
Therefore, if $n=0$, there is nothing to prove. For $n\geq1$,
we differentiate both sides of \eqref{eqHermite} $n-1$ times and use \eqref{eqDk} to obtain
\begin{multline*}
    a^{n-1} \inr \left[n!\,c_0x+(n-1)!\,c_1\right] e^{-\mu \left(x-az\right)^2}dx\\
    =b^{n-1}C_n \inr \left[n!\,c_0x+(n-1)!\,c_1\right] e^{-\nu \left(x-bz\right)^2}dx.
\end{multline*}
Lemma~\ref{lemcal} shows that both sides of the above equation are first degree polynomials of $z$.
If we compare the constant terms of these two linear polynomials, it follows from the assumption $a\neq b$
and Lemma~\ref{lemcal} that
\begin{equation}\label{eqc1}
    c_1=0.
\end{equation}
Thus condition (c) holds for $n=1$. For $n\geq2$ we use \eqref{eqc1} to write
$$P_n(x)=c_0x^n+c_2x^{n-2}+\cdots+c_n.$$
Differentiating both sides of \eqref{eqHermite} $n-2$ times and using \eqref{eqDk}, we obtain
\begin{multline*}
    a^{n-2} \inr \left[\frac{n!\,c_0}{2}x^2+(n-2)!c_2\right] e^{-\mu \left(x-az\right)^2}dx\\
    =b^{n-2}C_n \inr \left[\frac{n!\,c_0}{2}x^2+(n-2)!c_2\right] e^{-\nu \left(x-bz\right)^2}dx.
\end{multline*}
Comparing the constant terms on both sides of the above equation and using Lemma~\ref{lemcal} and \eqref{eqCn}, we arrive at
$$\frac{n!\,c_0}{4\mu}+(n-2)!\,c_2=\frac{a^2}{b^2}\left(\frac{n!\,c_0}{4\nu}+(n-2)!\,c_2\right).$$
This together with \eqref{eqHermieeq} and the assumption $a^2\not=b^2$ yields
$$n!\,c_0=-4\delta(n-2)!\,c_2.$$
Therefore, condition (c) holds for $n=2$ as well.

When $n\geq3$, we will prove (c) by induction on $k$. Thus we assume
$$c_{2k-1}=0\qquad {\rm and} \qquad (n-2k+2)!\,c_{2k-2}=-4\delta k(n-2k)!\,c_{2k}$$
whenever $2k<n$. Then we have
$$P_n(x)=\sum_{j=0}^{k}c_{2j}x^{n-2j}+c_{2k+1}x^{n-2k-1}+\cdots+c_n.$$
Differentiating both sides of \eqref{eqHermite} $n-2k-1$ times and using \eqref{eqDk}, we obtain{\small
\begin{align*}
    &a^{n-2k-1} \inr \left[\sum_{j=0}^{k}\frac{(n-2j)!\,c_{2j}x^{2k-2j+1}}{(2k-2j+1)!}+(n-2k-1)!\,c_{2k+1}\right] e^{-\mu \left(x-az\right)^2}dx\\
    &=b^{n-2k-1}C_n \inr \left[\sum_{j=0}^{k}\frac{(n-2j)!\,c_{2j}x^{2k-2j+1}}{(2k-2j+1)!}+(n-2k-1)!\,c_{2k+1}\right] e^{-\nu \left(x-bz\right)^2}dx.
\end{align*}}
Note that $2k-2j+1$ is odd for all $0\leq j\leq k$.
Comparing the constant terms on both sides of the above equation and using Lemma~\ref{lemcal} again,
we arrive at
$$b^{2k+1}(n-2k-1)!\,c_{2k+1}=a^{2k+1}(n-2k-1)!\,c_{2k+1}.$$
Since $a^{2k+1}\neq b^{2k+1}$, we must have $c_{2k+1}=0$. Thus
$$P_n(x)=\sum_{j=0}^{k}c_{2j}x^{n-2j}+c_{2k+2}x^{n-2k-2}+\cdots+c_n.$$
If $2k+1<n$, we continue this process by differentiating both sides of \eqref{eqHermite} $n-2k-2$ times to
obtain{\small
\begin{align*}
    &a^{n-2k-2} \inr \left[\sum_{j=0}^{k}\frac{(n-2j)!\,c_{2j}x^{2k-2j+2}}{(2k-2j+2)!}+(n-2k-2)!\,c_{2k+2}\right] e^{-\mu \left(x-az\right)^2}dx\\
    &=b^{n-2k-2}C_n \inr \left[\sum_{j=0}^{k}\frac{(n-2j)!\,c_{2j}x^{2k-2j+2}}{(2k-2j+2)!}+(n-2k-2)!\,c_{2k+2}\right] e^{-\nu \left(x-bz\right)^2}dx.
\end{align*}}
Again, comparing the constant terms on both sides of the above equation, we get
\begin{align}\label{eq2k}
    &b^{2k+2} \left[\sum_{j=0}^{k}\frac{(n-2j)!\,c_{2j}}{(2k-2j+2)!}\frac{(2k-2j+1)!!}{2^{k-j+1}\mu^{k-j+1}}+(n-2k-2)!\,c_{2k+2}\right]\nonumber\\
    &=a^{2k+2} \left[\sum_{j=0}^{k}\frac{(n-2j)!\,c_{2j}}{(2k-2j+2)!}\frac{(2k-2j+1)!!}{2^{k-j+1}\nu^{k-j+1}}+(n-2k-2)!\,c_{2k+2}\right].
\end{align}
From the induction hypothesis we deduce that
\begin{multline*}
\frac{(n-2j+2)!\,c_{2j-2}}{(2k-2j+4)!}\,\frac{(2k-2j+3)!!}{2^{k-j+2}\mu^{k-j+2}}\\
=-\frac{j\delta}{(k-j+2)\mu}\left[\frac{(n-2j)!\,c_{2j}}{(2k-2j+2)!}\frac{(2k-2j+1)!!}{2^{k-j+1}\mu^{k-j+1}}\right]
\end{multline*}
for any $j=1,\cdots, k$. Consequently,{\small
\begin{align*}
&\sum_{j=0}^{k}\frac{(n-2j)!\,c_{2j}}{(2k-2j+2)!}\frac{(2k-2j+1)!!}{2^{k-j+1}\mu^{k-j+1}}-\frac{(n-2k)!\,c_{2k}}{2(2k+2)\delta}\\
&=-\frac{(n-2k)!\,c_{2k}}{2(2k+2)\delta}\left[1-\frac{(k+1)\delta}{\mu}+\frac{(k+1)k\delta^2}{2\mu^2}
    +\cdots+(-1)^{k+1}\frac{(k+1)k\cdots 1\, \delta^{k+1}}{1\cdot2\cdots (k+1)\mu^{k+1}}\right]\\
    &=-\frac{(n-2k)!\,c_{2k}}{2(2k+2)\delta}\left(1-\frac{\delta}{\mu}\right)^{k+1}.
\end{align*}}
By \eqref{eqHermieeq}, we have
$$\left(1-\frac{\delta}{\mu}\right)b^{2}=\left(1-\frac{\delta}{\nu}\right)a^{2}.$$
Combining this with \eqref{eq2k}, we get
\begin{multline*}
    b^{2k+2} \left[\frac{(n-2k)!\,c_{2k}}{2(2k+2)\delta}+(n-2k-2)!\,c_{2k+2}\right] \\
    =a^{2k+2}\left[\frac{(n-2k)!\,c_{2k}}{2(2k+2)\delta}+(n-2k-2)!\,c_{2k+2}\right].
\end{multline*}
Since $a^{2k+2}\neq b^{2k+2}$, we must have
$$(n-2k)!\,c_{2k}=-2\delta(2k+2)(n-2k-2)!\,c_{2k+2},$$
this completes the induction argument.

To finish the proof, we assume that condition (c) holds. Recall that the leading term of the polynomial
$P_n(x)$ is $c_0 x^n$. If we let $C =c_0 2^{-n}\neq 0$, then it is easy to see that
$$P_n(x)=C\sum_{k=0}^{[\frac{n}{2}]}\frac{(-1)^k2^nn!}{4^kk!(n-2k)!\delta^k}x^{n-2k}.$$
A straightforward calculation shows that
\begin{align*}
&\frac{1}{C}\left[2 xP_n(x)-\frac{1}{\delta} P^{\prime}_n(x)\right]-(2x)^{n+1}\\
&=\sum_{k=1}^{[\frac{n}{2}]}\frac{(-1)^k2^{n+1}n!}{4^kk!\,(n-2k)!\,\delta^k}x^{n-2k+1}
 -\sum_{k=1}^{[\frac{n}{2}]+1}\frac{(-1)^{k-1}2^nn!\,(n-2k+2)}{4^{k-1}(k-1)!(n-2k+2)!\,\delta^{k}}x^{n-2k+1}\\
 &=\sum_{k=1}^{[\frac{n}{2}]}\frac{(-1)^k2^{n+1}(n+1)!}{4^kk!\,(n-2k+1)!\,\delta^k}x^{n-2k+1}
   -\frac{(-1)^{[\frac{n}{2}]}2^nn!\left(n-2[\frac{n}{2}]\right)}{4^{[\frac{n}{2}]}\left([\frac{n}{2}]\right)!
   \left(n-2[\frac{n}{2}]\right)!\,\delta^{[\frac{n}{2}]+1}}x^{n-2[\frac{n}{2}]-1}.
\end{align*}
Consequently, if $n$ is even, then
\begin{align*}
2 xP_n(x)-\frac{1}{\delta} P^{\prime}_n(x)&=C\sum_{k=0}^{[\frac{n}{2}]}\frac{(-1)^k2^{n+1}(n+1)!}
{4^kk!\,(n-2k+1)!\delta^k}x^{n-2k+1}\\
&=P_{n+1}(x),
\end{align*}
and if $n$ is odd, then
\begin{align*}
2 xP_n(x)-\frac{1}{\delta} P^{\prime}_n(x)
&=C\sum_{k=0}^{[\frac{n}{2}]}\frac{(-1)^k2^{n+1}(n+1)!}{4^kk!\,(n-2k+1)!\,\delta^k}x^{n-2k+1}-\frac{C(-1)^{\frac{n-1}{2}}2^nn!}
{4^{\frac{n-1}{2}}\left(\frac{n-1}{2}\right)!\delta^{\frac{n+1}{2}}}\\
&=C\sum_{k=0}^{[\frac{n+1}{2}]}\frac{(-1)^k2^{n+1}(n+1)!}{4^kk!\,(n-2k+1)!\,\delta^k}x^{n-2k+1}\\
&=P_{n+1}(x).
\end{align*}
This proves that
condition (c) implies (a).
\end{proof}

Recall that for any nonnegative integer $n$, the $n$th Hermite polynomial $H_n(x)$ is
defined by
$$H_n(x)=(-1)^2 e^{x^2} \frac{d^n}{dx^n} e^{-x^2}.$$
In general, it is easy to check that each $H_n(x)$ has degree $n$ and
$$H_{n}(x)=2xH_{n-1}(x)-H^{\prime}_{n-1}(x),\qquad n\geq1,$$
which can be used to compute $H_n(x)$ inductively. Obviously, to show that Theorem E holds, simply take
$\delta=1$ and apply Theorem~\ref{thmHit}. Furthermore, as a by-product of the proof of
Theorem~\ref{thmHit}, we will also obtain the explicit formula for the integral that appeared
in \eqref{eqHermite}.

\begin{corollary}
Let $\mu, a\in\C$ with $\re \mu>0$, and let $P_n(x)$ be the $n$ degree polynomial defined
in \eqref{eqDeHe}. Then
$$\inr P_n(x)e^{-\mu \left(x-az\right)^2}dx=C\sum_{k=0}^{[\frac{n}{2}]}\frac{(-1)^k2^nn!\,a^{n-2k}
\sqrt{\pi}}{4^kk!\,(n-2k)!\delta^{k}\sqrt{\mu}}\left(1-\frac{\delta}{\mu}\right)^k\, z^{n-2k}$$
for all $z\in\C$.
\end{corollary}

\begin{proof}
Write $$Q_n(z)=\inr P_n(x)e^{-\mu \left(x-az\right)^2}dx.$$
Since $c_{2k+1}=0$ whenever $2k+1\leq n$, it follows from Lemma~\ref{lemcal} that
$$Q_n(z)=\sum_{k=0}^{[\frac{n}{2}]}\frac{Q_n^{(n-2k)}(0)}{(n-2k)!}\,z^{n-2k}.$$
Furthermore, by the proof of Theorem~\ref{thmHit}, we have
\begin{align*}
    Q_n^{(n-2k)}(0)&=\frac{d^{n-2k}}{dz^{n-2k}}\left(\inr P_n(x)e^{-\mu \left(x-az\right)^2}dx\right)\bigg|_{z=0}\\
    &=a^{n-2k}\inr P_n^{(n-2k)}(x)e^{-\mu \left(x-az\right)^2}dx \bigg|_{z=0}\\
    &=-a^{n-2k}\frac{(n-2k+2)!\,c_{2k-2}\sqrt{\pi}}{4k\delta\sqrt{\mu}}\left(1-\frac{\delta}{\mu}\right)^{k} \\
    &=C a^{n-2k}\frac{(-1)^{k}2^nn!}{4^{k}k!\,\delta^{k}}\frac{\sqrt{\pi}}{\sqrt{\mu}}\left(1-\frac{\delta}{\mu}\right)^{k}.
\end{align*}
whenever $k\in\mathbb{N}$ with $2k\leq n$. This proves the desired result.
\end{proof}

We return to the calculation of eigenvalues and eigenvectors of the operators $T^{(s,t)}$.
Recall from the analysis at the beginning of this section that if $|\re s|<1$ and $|s|^2=1+|t|^2$ then
$T^{(s,t)}$ has an eigenvalue $\lambda_0$ with corresponding eigenvector $e^{\gamma z^2/2}$, where
$\lambda_0=\frac1{\sqrt s}\,\sqrt{\frac s{1+\gamma\overline t}}$ for a certain choice of $\gamma$ in the
unit disk that is guaranteed by Lemma~\ref{existenceofgamma}. With the help of Theorem~\ref{thmHit},
we are going to obtain additional eigenvalues and corresponding eigenvectors for $T^{(s,t)}$ in the
folowing result.

\begin{theorem}\label{thmev}
Suppose $(s,t)\in\C^2$ with $|s|^2=|t|^2+1$ and $|\re s|<1$. Let $\gamma$ be the unique number
in $\D$ from Lemma~\ref{existenceofgamma}. Then for each nonnegative integer $n$ the complex number
$$\lambda_n=\frac{1}{\sqrt{s}}\,\sqrt{\frac{s}{s+\overline{t}\gamma}}\,\frac{1}{(s+\overline{t}\gamma)^n}\in\T$$
is an eigenvalue of $T^{(s,t)}$ and the function $Q_n(z) e^{\gamma z^2/2}$
is a corresponding eigenvector, where
$$Q_n(z)=\inr  H_n\left(\frac{x}{\rho}\right)\exp\left[-\frac{2}{1+\gamma} \left(x-\frac{1+\gamma}{2}z\right)^2
\right]\,dx$$
is a polynomial of degree $n$. Here $H_n(x)$ is the $n$th Hermite polynomial and $\rho$ is a positive
number such that
$$\rho^{2}=\frac{(1+\gamma)\left[(s-\overline{t})(s+\overline{t}\gamma)-1\right]}{2\left[(s+\overline{t}\gamma)^2-1\right]}.$$
\end{theorem}

\begin{proof}
Since $|\gamma|<1$, we use Fubini's theorem and \eqref{eqIC} to obtain
\begin{align*}
    &T^{(s,t)}[Q_n(z) e^{\frac{\gamma}{2}z^2}](z)\\
    &=\frac{1}{\sqrt{s}}\inc Q_n(w)\exp\left[\frac{\gamma}{2}w^2+\frac{t}{2s}z^2-\frac{\overline{t}}{2s}\overline{w}^2+\frac{ z\overline{w}}{s}\right]\,d\lambda(w)\\
    &=\frac{e^{\frac{t}{2s}z^2}}{\sqrt s}
    \inr H_n\left(\frac{x}{\rho}\right)e^{-\frac{2}{1+\gamma} x^2}\,dx
    \inc\exp\left[-\frac{w^2 }{2}+2xw-\frac{\overline{t}}{2s}\overline{w}^2+\frac{ z\overline{w}}{s}
    \right]\,d\lambda(w)\\
    &=Ce^{\frac{t-\overline{s}}{2s(s-\overline{t})}z^2}
    \inr H_n\left(\frac{x}{\rho}\right)\exp\left[-\left(\frac{2}{1+\gamma}+\frac{2\overline{t}}{s-\overline{t}}\right) x^2
    +\frac{2z}{s-\overline{t}} x\right]\,dx\\
    &=Ce^{\frac{t+\overline{s}\gamma}{2(s+\overline{t}\gamma)}z^2}\inr  H_n\left(\frac{x}{\rho}\right)
    \exp\left[-\frac{2(s+\overline{t}\gamma)}{(1+\gamma)(s-\overline{t})} \left(x-\frac{1+\gamma}{2(s+\overline{t}\gamma)}z\right)^2\right]\,dx,
\end{align*}
where
$$C=\frac{1}{\sqrt{s}}\sqrt{\frac{s}{s-\overline{t}}}.$$
It follows from $|s|^2=1+|t|^2$ that the equation in (\ref{choiceofgamma}) can be rewritten as
\begin{equation}
\gamma(s+\overline{t}\gamma)=t+\overline{s}\gamma,
\label{choiceofgamma2}
\end{equation}
which implies
$$e^{\frac{t+\overline{s}\gamma}{2(s+\overline{t}\gamma)}z^2}=e^{\frac{\gamma}{2}z^2}.$$
So it suffices for us to prove that
\begin{multline}\label{eqintH}
    \frac{1}{\sqrt{s}}\sqrt{\frac{s}{s-\overline{t}}} \inr H_n\left(\frac{x}{\rho}\right)
    \exp\left[-\frac{2(s+\overline{t}\gamma)}{(1+\gamma)(s-\overline{t})} \left(x-\frac{1+\gamma}{2(s+\overline{t}\gamma)}z\right)^2\right]\,dx\\
    =\lambda_n \inr  H_n\left(\frac{x}{\rho}\right)\exp\left[-\frac{2}{1+\gamma}
    (x-\frac{1+\gamma}{2}z)^2\right]\,dx.
\end{multline}
It follows from $|\re s|<1$ and $|s|^2=1+|t|^2$ that $|\im t|<|\im s|$. This together with
\eqref{choiceofgamma2} and \eqref{eqgamma} implies that
\begin{align*}
\rho^{2}&=\frac{(s-\overline{t})(s+\overline{t}\gamma)+(s-\overline{t})(t+\overline{s}\gamma)-1-\gamma}
{2[s(s+\overline{t}\gamma)+\overline{t}(t+\overline{s}\gamma)-1]}\\
&=\frac{(s+\overline{t}\gamma)(s+t-\overline{s}-\overline{t})}
{2[(s+\overline{s})(s+\overline{t}\gamma)-2]}\\
&=\frac{(s+t-\overline{s}-\overline{t})}{2\left(s+\overline{s}-\frac{2}{s+\overline{t}\gamma}\right)}>0.
\end{align*}
It is easy to check that
$$\re \left(\frac{s}{s-\overline{t}}\right)>0 \qquad {\rm and} \qquad \re \left(\frac{s}{s+\overline{t}\gamma}\right)>0.$$
Thus \eqref{eqintH} is equivalent to
\begin{multline}\label{eqintH2}
     \inr H_n(x)\exp\left[-\frac{2(s+\overline{t}\gamma)\rho^2}{(1+\gamma)(s-\overline{t})}
     \left(x-\frac{1+\gamma}{2(s+\overline{t}\gamma)\rho}z\right)^2\right]\,dx\\
    =\sqrt{\frac{s-\overline{t}}{s+\overline{t}\gamma}}\frac{1}{(s+\overline{t}\gamma)^n}
    \inr H_n(x)\exp\left[-\frac{2\rho^2}{1+\gamma} (x-\frac{1+\gamma}{2\rho}z)^2\right]\,dx.
\end{multline}

Write
$$\mu=\frac{2(s+\overline{t}\gamma)\rho^2}{(1+\gamma)(s-\overline{t})},\quad
\nu=\frac{2\rho^2}{1+\gamma},\quad a=\frac{1+\gamma}{2(s+\overline{t}\gamma)\rho},
\quad b=\frac{1+\gamma}{2\rho}.$$
Then a straightforward calculation shows that $\re \nu>0$ and
\begin{align*}
\re \mu   &=\frac{2\rho^2(|s|^2-\re(st)+\re(\gamma)+\re(st)|\gamma|^2-|t|^2|\gamma|^2)}
    {|1+\gamma|^2|s-\overline{t}|^2}\\
&=\frac{\rho^2\left[(1+|s-\overline{t}|^2)+2\re(\gamma)+(1-|s-\overline{t}|^2)|\gamma|^2\right]}
{|1+\gamma|^2|s-\overline{t}|^2}\\
&> \frac{\rho^2}{|s-\overline{t}|^2}>0.
\end{align*}
Recall from the proof of Lemma~\ref{existenceofgamma} that $(s+\overline{t}\gamma)^2\neq1$.
Since $$\frac{(1+\gamma)(s-\overline{t})(s+\overline{t}\gamma)}{2\rho^{2}}-(s+\overline{t}\gamma)^2=\frac{1+\gamma}{2\rho^{2}}-1,$$
it follows that
$$\nu b^2\neq\mu a^2\qquad {\rm and} \qquad \frac{(b^2-a^2)\mu\nu}{\nu b^2-\mu a^2}=1.$$
Then by the proof of Theorem~\ref{thmHit}, the $n$th Hermite polynomial $H_n(x)$
satisfies \eqref{eqintH2}. This complete the proof of the theorem.
\end{proof}


Some special cases are worth mentioning here. If $(s+\overline{t}\gamma)^k=1$ for some $2<k\leq n$, then by the proof of Theorem~\ref{thmHit}, each Hermite polynomial $H_{n+kN}(x)$, $N=0,1,\cdots,$ satisfies \eqref{eqintH2}, which implies that all functions $Q_{n+kN}(z)e^{\gamma z^2/2}$ are eigenvectors for the
operator $T^{(s,t)}$ corresponding to the eigenvalue $\lambda_n$.
Conversely, if $(s+\overline{t}\gamma)^k\neq1$ whenever $1\leq k\leq n$, then 
$H_n(x)$ is the unique polynomial of degree less than or equal to $n$ satisfying \eqref{eqintH2}. Furthermore, if $(s+\overline{t}\gamma)^k\neq1$ for any positive number $k$, then we can completely determine the spectrum of the operator $T^{(s, t)}$. More specifically, we have the following.

\begin{corollary}
Suppose $(s,t)\in\C^2$ with $|s|^2=1+|t|^2$ and $|\re s|<1$. Let $\gamma$ be the number from
Lemma~\ref{existenceofgamma} and $1/(s+\overline t\gamma)=e^{i\theta}$, $\theta\in (-\pi, \pi]$.
If $\theta$ is not a rational multiple of $\pi$, then the spectrum of the unitary operator $T^{(s,t)}$ is
the full unit circle.
\end{corollary}

\begin{proof}
If $\theta$ is not a rational multiple of $\pi$, it is well known that the sequence
$\{e^{in\theta}: n=0,1,2,\cdots\}$ is dense in the unit circle $\T$. Since the spectrum $\sigma(T^{(s,t)})$
of the unitary operator $T^{(s,t)}$ is a closed subset of $\T$, it follows from Theorem~\ref{thmev} that
$\sigma(T^{(s,t)})=\T$.
\end{proof}

\section{Linear canonical transforms}

Recall that the general linear group of order $n$, denoted by $GL(n,\mathbb{R})$, is the group of all invertible
$n\times n$ matrices with real entries. The special linear group of order $n$, denoted by $SL(n, \mathbb{R})$,
is the subgroup of $GL(n, \R)$ consisting of matrices $A$ with $\det(A)=1$.

We will be mostly interested in the case when $n=2$. Instead of using four real entries to describe an
element in $GL(2, \R)$, it will be more convenient for us to think of a matrix in $GL(2,\R)$ as a pair
of complex numbers.

\begin{lemma}
Let $GL(\C\times\C)=\left\{(s, t)\in\C\times\C: |s|\neq|t|\right\}$.
Then $GL(\C\times C)$ is a group with the following operation:
   $$(s_1, t_1)\cdot(s_2, t_2)=(s_1s_2+\overline t_1t_2, t_1s_2+\overline s_1t_2).$$
Furthermore, $SL(\C\times\C)=\left\{(s, t)\in\C\times\C: |s|^2-|t|^2=1\right\}$
is a subgroup of $GL(\C\times\C)$.
\end{lemma}

\begin{proof}
The proof follows easily from the definitions. We omit the routine details. This will also follow from
Theorem~\ref{thmG} and its proof below.
\end{proof}

It is clear that the group $GL(\C\times\C)$ has unit $(1,0)$ and it is easy to check that the inverse of
$(s,t)$ is given by
$$(s,t)^{-1}=\left(\frac{\overline{s}}{|s|^2-|t|^2},-\frac{t}{|s|^2-|t|^2}\right).$$

\begin{theorem}\label{thmG}
Let
$$\varphi(A)=\left(\frac{a+bi+d-ci}{2}, \frac{a+bi-d+ci}{2}\right)$$
for $A=\left[ {\begin{array}{*{20}c}
   a & b  \\
   c & d  \\
\end{array}} \right]\in GL(2, \mathbb{R})$. Then $\varphi: GL(2,\R)\to GL(\C\times\C)$
is a group isomorphism, and it maps $SL(2,\R)$ onto $SL(\C\times\C)$.
\end{theorem}

\begin{proof}
Since
\begin{equation}\label{eqdet}
    \left|\frac{a+bi+d-ci}{2}\right|^2-\left|\frac{a+bi-d+ci}{2}\right|^2=ad-bc,
\end{equation}
we conclude that $\varphi(A)\in GL(\mathbb{C}\times\mathbb{C})$ for any $A\in GL(2, \mathbb{R})$.

It is clear that $\varphi$ is injective. Moreover, for any $(s, t)\in GL(\C\times\C)$, if we let
$$A=\left[ {\begin{array}{*{20}c}
   \re\ (s+t) & \im\ (s+t)  \\
   -\im\ (s-t) & \re\ (s-t)  \\
\end{array}} \right],$$
then $A\in GL(2, \R)$ and $\varphi(A)=(s, t)$. Thus $\varphi$ is also surjective.

A direct calculation shows that
\begin{align*}
\varphi(A_1A_2)=&\left(\frac{a_1a_2+b_1c_2+(a_1b_2+b_1d_2)i+c_1b_2+d_1d_2-(c_1a_2+d_1c_2)i}{2}
\right.,\\
&\quad\left.\frac{a_1a_2+b_1c_2+(a_1b_2+b_1d_2)i-(c_1b_2+d_1d_2)+(c_1a_2+d_1c_2)i}{2}\right)\\
=&\left(\frac{a_1+b_1i+d_1-c_1i}{2}, \frac{a_1+b_1i-d_1+c_1i}{2}\right)\\
&\quad\cdot\left(\frac{a_2+b_2i+d_2-c_2i}{2}, \frac{a_2+b_2i-d_2+c_2i}{2}\right)\\
=&\varphi(A_1)\cdot \varphi(A_2),
\end{align*}
which shows that $\varphi$ preserves group operations in $GL(2,\R)$ and $GL(\C\times\C)$. Thus
$\varphi$ is an isomorphism from $GL(2,\R)$ onto $GL(\C\times\C)$. It follows from \eqref{eqdet} that
$\varphi$ maps $SL(2, \R)$ onto $SL(\C\times\C)$.
\end{proof}

The following result gives a new unitary projective representation of $SL(\C\times\C)$, and hence $SL(2, \R)$,
on the Fock space. Recall that a mapping $\phi: G\to B(H)$ from a group $G$ to the algebra $B(H)$ of all
bounded linear operators on a Hilbert space $H$ is called a unitary projective (or ray) representation if
\begin{itemize}
\item[(i)] for every $x\in G$ the operator $\phi(x)$ is unitary, and
\item[(ii)] for any $x_1,x_2\in G$ there exists a unimodular constant $\lambda$ such that
$\phi(x_1x_2)=\lambda\phi(x_1)\phi(x_2)$.
\end{itemize}
See \cite{B0, Mackey} for the theory of unitary projective or ray representations.

\begin{theorem}\label{thmT}
The mapping $(s, t)\mapsto T^{(s,t)}$ is a unitary projective representation of the group $SL(\C\times\C)$
on $F^2$.
\end{theorem}

\begin{proof}
For $k=1,2$ let $(s_k, t_k)\in SL(\C\times\C)$, so $|s_k|^2=|t_k|^2+1$. It follows from Lemma~\ref{lem2} that
$$\inc K^{(s_1,t_1)}(z,w)K^{(s_2,t_2)}(w,u)\,d\lambda(w)=C_{(s,t)}K^{(s_1,t_1)\cdot(s_2,t_2)}(z,u),$$
where
$$C_{(s,t)}=\frac{\sqrt{s_1s_2+{\overline{t_1}t_2}}}{\sqrt{s_{1}}\sqrt{s_{2}}}
\sqrt{\frac{s_{1}s_{2}}{s_1s_2+{\overline{t_1}t_2}}}=\pm1.$$
Therefore,
\begin{eqnarray*}
    T^{(s_1, t_1)}T^{(s_2,t_2)} f(z)&=&\inc f(u)\,d\lambda(u)\inc K^{(s_1,t_1)}(z,w)
    K^{(s_2,t_2)}(w,u)\,d\lambda(w)\\
    &=&C_{(s,t)}\inc f(u)K^{(s_1,t_1)\cdot(s_2,t_2)}(z,u)\,d\lambda(u)\\
    &=&C_{(s,t)}T^{(s_1,t_1)\cdot(s_2,t_2)}f(z).
\end{eqnarray*}
By Theorem~\ref{thmTB}, each $T^{(s,t)}$ is a unitary operator on $F^2$, so
$(s,t)\mapsto T^{(s,t)}$ is a unitary projective representation of $SL(\C\times\C)$ on $F^2$.
\end{proof}

Next we will consider the Hilbert space $L^2(\R)=L^2(\R,dx)$ and a family of unitary operators on it, namely,
the so-called linear canonical transforms that were traditionally studied in Hamiltonial mechanics. These
transforms include the fractional Fourier transforms, the Fresnel transform, as well as many other classical
transforms on $L^2(\R)$. The books \cite{G, HKOS, OZK, W} and survey paper
\cite{BM2} are excellent sources of information for these operators.

The linear canonical transforms can be defined in several ways. One way is to define them as parameterized
by the special linear group $SL(2, \R)$. More specifically, for any real matrix
$$A=\left[ {\begin{array}{*{20}c}
   a & b  \\
   c & d  \\
\end{array}} \right],\ \ \ \det(A)= ad-bc=1, $$
we define an operator $\F^A$ on $L^2(\R)$ by
\begin{equation}\label{eq1}
\F^A(f)(x)=\frac{1}{\sqrt{i\pi b}}\,e^{idx^2/b}\inr e^{-i(2xt-at^2)/b}f(t)\,dt
\end{equation}
for $b\neq0$. When $b=0$, we define
$$\F^A(f)(x)=\sqrt{d}\,e^{icdx^2}f(dx),$$
which is the limit in (\ref{eq1}) as $b\to0$. Note that the symbol $d$ here is a number but not the differential!

The linear canonical transforms include several prominent classical transforms as special cases. First,
for any angle $\alpha\in[-\pi,\pi]$, the matrix
$$A_\alpha=\left[\begin{matrix}\cos\alpha & \sin\alpha\\ -\sin\alpha & \cos\alpha\end{matrix}\right]$$
is clearly a member of $SL(2, \R)$. The corresponding linear canonical transform is given by
$$\F^{A_\alpha}f(x)=\frac{1}{\sqrt{i\pi\sin\alpha}}\,e^{ix^2\cot\alpha}\inr e^{-i(2xt\csc\alpha
-t^2\cot\alpha)}f(t)\,dt$$
when $\sin\alpha\not=0$, and
$$\F^{A_\alpha}f(x)=\sqrt{\cos\alpha}\,e^{-ix^2\sin\alpha\cos\alpha}f(x\cos\alpha)=\sqrt{\pm1}\,f(\pm x)$$
when $\sin\alpha=0$. Up to a unimodular constant, these are the classical fractional Fourier transforms.
In fact, since
\begin{equation}\label{eqcot}
\frac{e^{i\alpha/2}}{\sqrt{i\pi\sin\alpha}}=\frac{\sqrt{1-i\cot\alpha}}{\sqrt{\pi}},
\end{equation}
the classical $\alpha$th order fractional Fourier transform $\mathcal{F}^\alpha$ (see \cite{BM2, OZK} for
example) can be written as $\mathcal{F}^\alpha=e^{i\alpha/2}\F^{A_\alpha}$ for any $\alpha\in[-\pi,\pi)$.
When $\alpha=\pi/2$, we obtain the ordinary Fourier transform.
The inverse Fourier transform is obtained when $\alpha=-\pi/2$.

For any $\sigma>0$, the matrix
$$A=\left[\begin{matrix}1/\sigma & 0\\ 0 & \sigma\end{matrix}\right]$$
belongs to $SL(2, \R)$ and the corresponding linear canonical transform is given by
$$\F^A(f)(x)=\sqrt\sigma\,f(\sigma x).$$
This is ``scaling'' or ``dilation'' in $L^2(\R)$.

Let $b={\lambda l}/{\pi}>0$ (where $l$ and $\lambda$ represent distance and wavelength, respectively, in
mechanics). The linear canonical transform corresponding to the matrix
$$A=\left[\begin{matrix} 1 & b\\ 0 & 1\end{matrix}\right]$$
is given by
$$\F^A(f)(x)=\frac{1}{\sqrt{i\pi b}}\inr e^{i(x-t)^2/b}f(t)\,dt.$$
This is basically the classical Fresnel transform $\F^A_l$, which corresponds to shearing in continuum
mechanics. More precisely, we have
$$\F_l^{A}(f)(x)=e^{i\pi l/\lambda}\F^A(f)(x).$$

Finally, for the matrix
$$A=\left[\begin{matrix}1 & 0\\ \tau & 1\end{matrix}\right]$$
in $SL(2, \R)$, the associated linear canonical transform is given by
$$\F^A(f)(x)=e^{i\tau x^2}f(x),$$
which is traditionally called the ``chirp multiplication'' in optics.

Our next goal is to show that every linear canonical transform is unitarily equivalent to some operator
$T^{(s,t)}$ on $F^2$. To this end, recall that the Bargmann transform $B$ is the operator from
$L^2(\R)\rightarrow F^2$ defined by
$$Bf(z)=C\inr f(x)e^{2xz-x^2-(z^2/2)}\,dx,$$
where $C=(2/\pi)^{1/4}$. It is well known that $B$ is a unitary operator from $L^2(\R)$ onto $F^2$.
Furthermore, the inverse of $B$ is also an integral operator, namely,
$$B^{-1}f(x)=C\inc f(z)e^{2x\overline{z}-x^2-(\overline{z}^2/2)}\,d\lambda(z).$$
See \cite{B1, B2, G, Zhu1}.

We consider the action of the Bargmann transform on linear canonical transforms. In other
words, we are going to compute the equivalent form, under the Bargmann transform,
of the linear canonical transforms on the Fock space.

\begin{theorem}\label{thmC}
Suppose
$$A=\left[\begin{matrix} a & b\\ c & d\end{matrix}\right]$$
is a matrix in $SL(2, \R)$. Then the operator $T^{A}=B \F^A  B^{-1}:F^2\rightarrow F^2$ is given by
$T^A=C_AT^{(s, t)}$, where
\begin{equation}\label{eqst}
    s=\frac{a+bi+d-ci}{2}, \qquad t=\frac{a+bi-d+ci}{2}
\end{equation}
with $(s,t)\in SL(\C\times\C)$, and
$$C_A=\pm1=\begin{cases}
    \displaystyle
\frac{\sqrt{s}}{\sqrt{s+t-\overline{s}-\overline{t}}}\sqrt{\frac{s+t-\overline{s}
-\overline{t}}{s}},&{\rm if}\ b \neq 0,\\[10 pt]
    \displaystyle
\sqrt{s}\sqrt{\frac{1}{s+t}}\sqrt{\frac{s+t}{s}},& {\rm if}\ b=0.\end{cases}$$
\end{theorem}

\begin{proof}
First we assume $b\neq0$. For the purpose of applying Fubini's theorem in the calculations below,
we assume that $f$ is any polynomial (recall that the polynomials are dense
in $F^2$, and under the inverse Bargmann transform, they become the Hermite polynomials times the Gauss function, which have very good integrability properties on the real line).
Recall that $C=(2/\pi)^{1/4}$ and let us write $C'=\frac{1}{\sqrt{ib\pi}}$ for simplicity. Then we have
\begin{align*}
    \F^A (B^{-1}f)(x)
    &=CC'e^{idx^2/b}\inr e^{-i(2xt-at^2)/b}\,dt\inc f(w)e^{2t\overline{w}-t^2-(\overline{w}^2/2)}\,d\lambda(w)\\
    &=CC'e^{idx^2/b}\inc f(w)\exp\left[-\frac{\overline{w}^2}{2}+\frac{(b\overline{w}-ix)^2}{b(b-ia)}\right]
    \,d\lambda(w)\\
    &\qquad\qquad\qquad\qquad\cdot\inr\exp\left[-\Bigl(1-i\frac{a}{b}\Bigr)
    \Bigl(t-\frac{b\overline{w}-i x}{b-ia}\Bigr)^2\right]\,dt.
\end{align*}
It follows from Lemma~\ref{lemcal} that{\small
\begin{align*}
   &\F^A (B^{-1}f)(x)\\
   &=C''e^{idx^2/b}
    \inc f(w)\exp\left[-\frac{\overline{w}^2}{2}+\frac{(b\overline{w}-ix)^2}{b(b-ia)}\right]\,d\lambda(w)\\
    &=C''\exp\left[\Bigl(\frac{id}b-\frac1{b(b-ia)}\Bigr)\,x^2\right]\inc f(w)
    \exp\left[\frac{b+ia}{2(b-ia)}\overline{w}^2-\frac{2i}{b-ia}x\overline{w}\right]\,d\lambda(w),
\end{align*}}
where $C''=CC'\sqrt{\frac{b\pi}{b-ia}}$.
Since $ad-bc=1$, a direct calculation shows that
$$(b-id)(b-ia)+1=b(b-ia-c-id).$$
Therefore,{\small
\begin{align*}
    &B \F^AB^{-1}f(z)\\
    &=CC''\inr\exp\left[2xz-\frac{z^2}{2}-\frac{(b-id)(b-ia)+1}{b(b-ia)}
       \right]dx\\
    &\qquad\qquad\qquad\cdot\inc f(w)\exp\left[\frac{b+ia}{2(b-i a)}\overline{w}^2
    -\frac{2i}{b-i a}x\overline{w}\right]\,d\lambda(w)\\
    &=CC''e^{-z^2/2}\inc f(w)\exp\left[\frac{b+ia}{2(b-ia)}\overline{w}^2+
    \frac{\left[(b-ia)z-i\overline{w}\right]^2}{(b-ia)(b-ia-c-id)}\right]\,d\lambda(w)\\
    &\qquad\qquad\qquad\cdot\inr \exp\left[-\frac{b-ia-c-id}{b-ia}
    \left(x-\frac{(b-ia)z-i\overline{w}}{b-ia-c-id}\right)^2\right]dx.
\end{align*}}
Observe that
$$
{\rm Re}\,\frac{b-ia-c-id}{b-ia}
=\frac{b^2+a^2+1}{b^2+a^2}>0.$$
Using Lemma~\ref{lemcal} again, we obtain
\begin{align*}
    B \F^AB^{-1}f(z)
    &=C_A\frac{\sqrt{2}}{\sqrt{a+bi-ci+d}}\,e^{-z^2/2}\\
    &\cdot\inc f(w)\exp\left[\frac{b+ia}{2(b-ia)}\overline{w}^2+
    \frac{\left[(b-ia)z-i\overline{w}\right]^2}{(b-ia)(b-ia-c-id)}\right]\,d\lambda(w)
\end{align*}
where $$C_A=\frac{\sqrt{a+bi-ci+d}}{\sqrt{ib}}\sqrt{\frac{ib}{a+bi+d-ci}}.$$
Then a few lines of elementary calculations show that
\begin{align*}
    &B\F^A  B^{-1}f(z)=C_A\frac{\sqrt{2}}{\sqrt{a+bi-ci+d}}\\
    &\qquad \cdot\inc f(w)\exp\left[\frac{(a+bi+ci- d)z^2-{(a-bi-ci-d)}\overline{w}^2
    +4 z\overline{w}}{2(a+ib-ci+ d)}\right]\,d\lambda(w)\\
    &=\frac{C_A}{\sqrt{s}}\inc f(w)e^{\frac{t}{2s}z^2-\frac{\overline{t}}{2s}\overline{w}^2+\frac{ z\overline{w}}{s}}\,d\lambda(w),
\end{align*}
where
$$s=\frac{a+bi+d-ci}{2},\qquad t=\frac{a+bi-d+ci}{2}.$$
It is easy to check that $|s|^2=|t|^2+1$, so $(s,t)\in SL(\C\times\C)$.

Next we assume $b=0$. Note that
$$d=\frac{1}{a}=\frac{1}{s+t}$$
The arguments used above can be simplified to show that
$$T^{A}f(z)=\sqrt{s}\,\sqrt{\frac{1}{s+t}}\,\sqrt{\frac{s+t}{s}}\,T^{(s, t)}f(z).$$
This completes the proof of the theorem.
\end{proof}

As a consequence of Theorem~\ref{thmC}, we immediately derive a number of basic properties for
the operators $\F^A$. In particular, we obtain an alternative proof of the unitarity and the composition
formula of linear canonical transforms.

\begin{corollary}\label{4}
Let $A, A_{1}, A_{2}\in SL(2, \mathbb{R})$. Then the following statements hold.
\begin{enumerate}
  \item[(a)]$\F^A$ is a unitary operator on $L^2(\R)$.
  \item[(b)] $\F^{A_{1}}\F^{A_{2}}=C\F^{A_{1}A_{2}}$ for some $C=\pm1$.
  \item[(c)] $\left(\F^A\right)^{-1}=C\F^{{A}^{-1}}$, 
  where $C=-1$ whenever $a<0$ and $b=0$, and $C=1$ for other cases.
  Here $A=\left[\begin{matrix} a & b\\ c & d\end{matrix}\right]\in SL(2, \mathbb{R})$.
\end{enumerate}
\end{corollary}

\begin{proof}
Condition (a) clearly follows from Theorems~\ref{thmTB} and \ref{thmC}.

Recall from Theorem~\ref{thmG} that $SL(2, \R)$ is isomorphic to $SL(\C\times\C)$. 
Therefore, it follows from Theorems~\ref{thmT} and \ref{thmC} that
\begin{eqnarray*}
B\F^{A_{1}}\F^{A_{2}}B^{-1}f(z)&=&C_{A_1}C_{A_2}T^{(s_1, t_1)}T^{(s_2, t_2)} f(z)\\
&=&C_{A_1}C_{A_2}C_{(s,t)}T^{(s_1, t_1)\cdot(s_2, t_2)} f(z)\\
&=&C\ B\F^{A_{1}A_{2}}B^{-1}f(z)
\end{eqnarray*}
for any $f\in F^2$, where $C=C_{A_1}C_{A_2}C_{A_1A_2}C_{(s,t)}=\pm1$. This proves (b).

A special case of assertion (b) is that
$$B\F^{A^{-1}}\F^A B^{-1} f(z)=CT^{(1, 0)}f(z)
=C\inc f(u)e^{\overline{w}z}\,d\lambda(w)=Cf(z),$$
where 
$$C=\frac{|s|}{\sqrt{s+t-\overline{s}-\overline{t}}}\sqrt{\frac{s+t-\overline{s}-\overline{t}}{s}}
\frac{1}{\sqrt{\overline{s}-t-s+\overline{t}}}\sqrt{\frac{\overline{s}-t-s+\overline{t}}{\overline{s}}}$$
when the real number $-i(s+t-\overline{s}-\overline{t})\neq 0$, and
$$C=|s|\sqrt{\frac{1}{s+t}}\sqrt{\frac{s+t}{s}}\sqrt{\frac{1}{\overline{s}-t}}\sqrt{\frac{\overline{s}-t}{\overline{s}}}$$
when $-i(s+t-\overline{s}-\overline{t})= 0$.

Suppose $-i(s+t-\overline{s}-\overline{t})\neq 0$. Then $(s+t-\overline{s}-\overline{t})/{s}<0$ if and only if
$$s=-\overline{s}\qquad {\rm and}\qquad [-i(s+t-\overline{s}-\overline{t})](-is)<0.$$
However, this condition means that
$$|s|^2=|-is|^2<|\im t|\leq |t|^2,$$
which contradicts with the fact that $(s,t)\in SL(\C\times\C)$.
Therefore,
$$\sqrt{\frac{s+t-\overline{s}-\overline{t}}{s}}\sqrt{\frac{\overline{s}-t-s+\overline{t}}{\overline{s}}}=
\frac{|s+t-\overline{s}-\overline{t}|}{|s|}.$$
It is clear that $s+t-\overline{s}-\overline{t}$ is not a negative number as well. So we have  $C=1$.

Now we consider the remaining case $-i(s+t-\overline{s}-\overline{t})=0$.
Since $(s,t)\in SL(\C\times\C)$, it is clear that both $s+t$ and $\overline{s}-t$ are nonzero real numbers, and
$$(s+t)(\overline{s}-t)=1>0.$$
Note that $$\re\left(\frac{s+t}{s}\right)>0\ \qquad {\rm and}\qquad \re\left(\frac{\overline{s}-t}{\overline{s}}\right)>0.$$
So we have $$C=\sqrt{\frac{1}{s+t}}\sqrt{\frac{1}{\overline{s}-t}}={\rm sgn} (s+t).$$
This proves (c) and completes the proof of the corollary.
\end{proof}

We encountered the possibility of an additional minus sign in the unitary representation (sometimes
called the composition formula in the literature) of the special linear group $SL(\C\times\C)$. This problem
is known as the metaplectic sign problem and is carefully studied in \cite[Sect. 9.1.4]{W}.

Recall that for each fixed $(s,t)\in SL(\C\times\C)$ with $|\re s|<1$ we found a sequence $\{\lambda_n\}$
of eigenvalues for the unitary operator $T^{(s,t)}: F^2\to F^2$ with the corresponding eigenfunctions
$$f_n(z)=Q_n(z) e^{\frac{\gamma}{2}z^2}, n=0,1,2,\cdots.$$
Under the inverse Bargmann transform, these eigenfunctions change as follows:
\begin{align*}
    B^{-1}(f_n)(x)&=C\inc\exp\left[2x\overline{z}-x^2-\frac{\overline{z}^2}{2}\right]\,d\lambda(z)\\
    &\qquad\cdot\inr H_n(t/\rho)\exp\left[-\frac{1-\gamma}{1+\gamma}t^2-\frac{1}{2}z^2 +2tz-t^2\right]dt\\
    &=CH_n(x/\rho)\exp\left[-\frac{1-\gamma}{1+\gamma}x^2\right],
\end{align*}
where $H_n$ are the Hermite polynomials and $C=(2/\pi)^{1/4}$.
In view of \eqref{eqst}, we have $|\re s|<1$ if and only if $|a+d|<2$.
See \cite{PD2} for related work on linear canonical transforms.

As another consequence of Theorem~\ref{thmC}, we obtain the equivalent form, via the Bargmann transform,
of several linear canonical transforms on $F^2$.

First consider the fractional Fourier transform, which is one of the most valuable and powerful tools in
mathematics, quantum mechanics, optics and signal processing. There are several well-accepted
normalizations for the fractional Fourier transform. We will chose the folowing definition.
For $\alpha\in[-\pi, \pi]$ (or for any real number $\alpha$), we define the $\alpha$th order fractional
Fourier transform by
$$\mathcal{F}^\alpha(f)(x)=\frac{\sqrt{1-i\cot\alpha}}{\sqrt{\pi}}e^{ix^2\cot\alpha}
\inr e^{-i(2xt\csc\alpha-t^2\cot\alpha)}f(t)\,dt.$$

Several special cases are worth mentioning. First, the case $\alpha=0$ needs to be interpretted as a
limit, and as such, it is just the identity operator (this is partial justification for the particular normalization
above). When $\alpha=\pm\pi$, we also need to interpret the corresponding transform as a limit, and
an elementary calculation shows that $\F^{\pm\pi}(f)(x)=f(-x)$, which is also called the parity transform.
The most prominent special cases are when $\alpha=\pm\,\pi/2$: $\F^{\pi/2}$ is the classical Fourier
transform and $\F^{-\pi/2}$ is the inverse Fourier transform. See \cite{BM2, OZK} for more information
about fractional Fourier transforms.

In the framework of linear canonical transforms, we have
\begin{equation}\label{eq2}
\F^\alpha=e^{i\alpha/2}\F^{A_\alpha},\qquad \alpha\in[-\pi,\pi),
\end{equation}
where
$$A_\alpha=\left[ {\begin{array}{*{20}c}
   \cos\alpha & \sin\alpha  \\
   -\sin\alpha & \cos\alpha  \\
\end{array}} \right].$$
This leads to the following result, which can be found in \cite{DZ}.

\begin{corollary}\label{3}
For any $\alpha\in[-\pi, \pi]$ the operator
$$T=B\mathcal{F}^\alpha  B^{-1}:F^2\rightarrow F^2$$ is given by $Tf(z)=f(e^{-i\alpha}z)$ for all $f\in F^2$. Consequently, the operator $$T^{-1}=B\left(\mathcal{F}^\alpha \right)^{-1}B^{-1}:F^2\rightarrow F^2,$$ where $\left(\mathcal{F}^\alpha \right)^{-1}$ is the inverse fractional Fourier transform, is given by $T^{-1}f(z)=f(e^{i\alpha}z)$ for all $f\in F^2$.
\end{corollary}

\begin{proof}
First notice that $\varphi(A)=(e^{i\alpha},0)$ for the real rotation matrix $A$ above in view of \eqref{eqst}.
So it follows from \eqref{eqcot}, \eqref{eq2} and Theorem~\ref{thmC} that
$$B\mathcal{F}^\alpha B^{-1}f(z)=\inc f(w)\exp\left[e^{-i\alpha}z\overline{w}\right]\,d\lambda(w)
=f(e^{-i\alpha}z).$$
It is then clear that $B\left(\mathcal{F}^\alpha \right)^{-1}B^{-1}f(z)=f(e^{i\alpha}z)$.
\end{proof}

Recall that for any positive $r$ the (weighted) dilation operator $D_r: L^2(\R)\longrightarrow L^2(\R)$ is
defined by $D_rf(x)=\sqrt{r}\,f(rx)$. We have $D_r=\F^{A_r}$ with
$$A_r=\left[\begin{matrix} 1/r & 0\\ 0 & r\end{matrix}\right],$$
and the linear canonical transform $\F^{A_r}$ in this case is also called {\it scaling}.

\begin{corollary}
For any positive $r$ the Bargmann transform takes the operator $D_r$ to the following unitarily equivalent
form on the Fock space:
\begin{align*}
T_rf(z)&=BD_rB^{-1}f(z)\\
&=\sqrt{\frac{2r}{1+r^2}}\inc f(w)\exp\left[\frac{2r}{1+r^2}z\overline{w}
+\frac{1- r^2}{2(1+r^2)}(z^2-\overline{w}^2)\right]\,d\lambda(w).
\end{align*}
\end{corollary}

\begin{proof}
This follows from Theorem~\ref{thmC} and the observation that
$$\varphi(A_r)=\left(\frac{1+r^2}{2r},\frac{1- r^2}{2r}\right).\qquad\qquad\qedhere$$
\end{proof}

For a positive parameter $b$, the matrix
$$A_b=\begin{bmatrix} 1 & b\\ 0 & 1\end{bmatrix}$$
in $SL(2, \R)$ gives rise to the following linear canonical transform:
$$\F^{A_b}(f)(x)=\frac{1}{\sqrt{i\pi b}}\inr e^{i(x-t)^2/b}f(t)\,dt.$$
This is called the Fresnel transform or the Gauss-Weierstrass transform. It is also called the chirp convolution.

\begin{corollary}
With the above notation, the Bargmann transform takes the operator $\\F^{A_b}$ on $L^2(\R)$ to the
following unitarily equivalent form on $F^2$:
\begin{align*}
T_bf(z)&=B\F^{A_b}B^{-1}f(z)\\
&=\sqrt{\frac{2}{2+i b}}\inc f(w)\exp\left[z\overline{w}+\frac{bi}{2(2+ bi )}(z-\overline{w})^2\right]
\,d\lambda(w).
\end{align*}
\end{corollary}

\begin{proof}
This follows from Theorem~\ref{thmC} and the fact that
$$\varphi(A_b)=\left(\frac{2+i b}{2},\frac{i b}{2}\right).\qquad\qquad\qedhere$$
\end{proof}

Finally in this section, recall that the chirp multiplication (or multiplication by a Gaussian) is the linear
canonical transform $\F^{A_c}$ on $L^2(\R)$ corresponding to
$$A_c=\begin{bmatrix} 1& 0\\ c & 1\end{bmatrix}.$$
More specifically, $\F^{A_c}(f)(x)=e^{i cx^2}f(x)$.

\begin{corollary}
The Bargmann transform sends the chirp multiplication $\F^{A_c}$ on $L^2(\R)$ above to the following
unitarily equivalent form on the Fock space $F^2$:
\begin{align*}
T_cf(z)&=B\F^{A_c}B^{-1}f(z)\\
&=\sqrt{\frac{2}{2-i c}}\inc f(w)\exp\left[z\overline{w}+\frac{ic}{2(2-i c)}
(z+\overline{w})^2\right]\,d\lambda(w).
\end{align*}
\end{corollary}

\begin{proof}
This is a consequence of Theorem~\ref{thmC} and the identity
$$\varphi(A_c)=\left(\frac{2-i c}{2},\frac{i c}{2}\right). \qquad\qquad\qedhere$$
\end{proof}

In the case of fractional Fourier transforms, the Bargmann transform takes a very complicated integral
transform on $L^2(\R)$ to an extremely simple operator on $F^2$. On the other hand, the Bargmann
transform takes the very simple dilation and chirp multiplication on $L^2(\R)$ to relatively complicated
integral operators on $F^2$. This is a good illustration of the usefulness of the Bargmann transform (and
its inverse): it can sometimes be used to simplify operators on $L^2(\R)$ and it can sometimes also be
used to simplify operators on $F^2$.

\end{document}